\documentclass[12pt,a4paper]{article}
\usepackage{amsmath,amsthm,amssymb, mathrsfs}
\usepackage{hyperref}
\usepackage{tikz}
\input xy
\xyoption{all}
\usepackage{epsfig}
\newtheorem{thm}{Theorem}[section]

\newtheorem{lem}[thm]{Lemma}

 \theoremstyle{definition}
\newtheorem{defin}[thm]{Definition}

\theoremstyle{remark}
\newtheorem{remark}[thm]{Remark}

\numberwithin{equation}{section}
\newcommand{\ben}{\begin{enumerate}}

\newcommand{\een}{\end{enumerate}}
\newcommand{\bit}{\begin{itemize}}
\newcommand{\eit}{\end{itemize}}

\begin{document}

\title
{Recovering Finite Parametric Distributions and Functions Using the Spherical Mean Transform}

\vskip 1cm
\author{Yehonatan Salman \\ Email: salman.yehonatan@gmail.com\\ Weizmann Institute of Science}
\date{}

\maketitle

\begin{abstract}

The aim of the article is to recover a certain type of finite parametric distributions and functions using their spherical mean transform which is given on a certain family of spheres whose centers belong to a finite set $\Gamma$. For this, we show how the problem of reconstruction can be converted to a Prony's type system of equations whose regularity is guaranteed by the assumption that the points in the set $\Gamma$ are in general position. By solving the corresponding Prony's system we can extract the set of parameters which define the corresponding function or distribution.

\end{abstract}

\section{Introduction and Motivation}

The aim of the article is to recover signals $f$ of the form
\begin{equation}
f(x) = \sum_{k = 1}^{m}a_{k}g(|x - x_{k}|), x_{k}\in\Bbb R^{n}, a_{k}\in\Bbb R\setminus\{0\}, 1\leq k\leq m,
\end{equation}
from the spherical mean transform (SMT for short) which integrates functions on spheres with a given set of centers and corresponding radii. Here, $g$ is a given scalar function defined on $\Bbb R^{+} = [0,\infty)$ and $m$ is a given fixed positive integer. We will also be interested in recovering signals $f$ of the form
\begin{equation}
f = \sum_{k = 1}^{m}a_{k}\delta_{x_{k}}, x_{k}\in\Bbb R^{n}, a_{k} \in\Bbb R\setminus\{0\}, 1\leq k\leq m,
\end{equation}
where $\delta_{x_{k}}$ is the shifted delta function to the point $x_{k}$, i.e., $\delta_{x_{k}}(x) = \delta(x - x_{k})$. The dual case, where $f$ is a sum of delta functions supported on hyperplanes rather than on points, will also be investigated.

Signals of the form (1.1)-(1.2) appear in many scientific fields such as in Signal Processing and Bioimaging (\cite{5, 6, 9, 20, 21}) and in other various mathematical fields such as in Inverse Problems and Approximation Theory (\cite{14, 28}).

Reconstruction of functions from their SMT is a well-known problem that has been investigated by many authors (\cite{1, 4, 10, 11, 15, 16, 18, 22, 23, 25, 26, 30}). The general problem can be described as finding a formula, or more generally an algorithm, for the reconstruction of a general function $f$, defined on $\Bbb R^{n}$, via its SMT.

Each sphere in $\Bbb R^{n}$ is determined by its center point $x\in\Bbb R^{n}$ and radius $r\geq0$ and thus the set of all spheres in $\Bbb R^{n}$ is $n + 1$ dimensional. Hence, the problem of reconstructing $f$ from its SMT is overdetermined since the space $\Bbb R^{n}$, on which $f$ is defined, is $n$ dimensional. Thus, in order to obtain a well posed problem one has to restrict the domain of definition of the SMT. In most cases it is assumed that the SMT is restricted to a set of the form $\Gamma\times\Bbb R^{+}$, where $\Gamma$ is a hypersurface in $\Bbb R^{n}$. That is, the centers of the spheres of integration are assumed to belong to a hypersurface while no restriction is imposed on the radii.

The problem of reconstructing a function $f$ from its SMT, restricted to such family of sets, arises in many practical fields such as thermo and photoacoustic tomography, radar and sonar imaging and approximation theory (\cite{2, 7, 19, 24, 27}). In the last few decades the reconstruction problem was solved in many cases such as where $\Gamma$ is a plane (\cite{4, 22}), a quadratic hypersurface (ellipsoid, paraboloid or a hyperboloid) (\cite{1, 10, 11, 15, 16, 18, 25}) or a cylinder (\cite{16,30}).
In all of the above obtained results there are no prior assumptions on the functions in question to be recovered other than some smoothness and support conditions. However, since each member in our family of functions has the form (1.1) it follows that it depends only on a finite set of parameters. Thus, we will have to shrink the set $\Gamma$ to a discrete set in order to obtain a well posed problem.

\begin{remark}

Observe that for a discrete set $\Gamma$, the set $\Gamma\times\Bbb R^{+}$ is one dimensional while each function $f$ of the form (1.1) depends only on a finite set of parameters (i.e., on a set of dimension $0$). However, one should observe that we cannot restrict the set of radii to be also discrete since otherwise the ability of reconstructing the function $f$ will depend on the function $g$. Indeed, if the set of radii is also discrete then the set of spheres on which the SMT is defined is at most countable. In this case it is not hard to choose a function $g$, with sufficiently small support near the origin, such that non of the spheres of integration will intersect the support of $f$. Thus, reconstruction of $f$ will be impossible in this case.

\end{remark}

The reconstruction problem of signals of the form (1.1)-(1.2) from the SMT arises in cases where one would like to recover point-wise signals with individual masses distributed in a relatively homogenous medium and where the data is collected from transducers scattered near the signals in question.

The main results of the article assert that if the set $\Gamma$ of centers of the spheres of integration consists of sufficiently many points in general position, then one can reconstruct signals $f$ of the form (1.1)-(1.2). The main idea behind these results is that the problem of reconstruction can be converted to the problem of solving a nonlinear Prony's type system of equations. The assumption that the points in the set $\Gamma$ are in general position guarantees the regularity of the obtained Prony's systems. We will also assume that the amplitudes $a_{1},...,a_{m}$ of $f$ are mutually distinct so that the solutions of the obtained Prony's systems can be used the get information on the distances between the points in $\Gamma$ and the translations $x_{1},...,x_{m}$. Using this information we can extract the points $x_{1},...,x_{m}$ and the amplitudes $a_{1},...,a_{m}$ which define the signals $f$.

\begin{remark}

Reconstructing signals of the form (1.1)-(1.2) using Prony's systems of equations is a known method that has been used, for example, in \cite{3, 5, 29}. There, the data for a signal in the form (1.1) or (1.2) was collected from a discrete set of values of its Fourier transform which results in a set of integral moments that can be converted to a Prony's type system of equations. Since we are dealing with the SMT rather than with the Fourier transform, our method should be modified accordingly.

\end{remark}

We will start by reconstructing signals $f$ of the form (1.2) and thus we will first define the SMT generally for distributions (in particular, this will define the SMT in the case where the signal $f$ is a function of the form (1.1)). Then, we will show that the obtained reconstruction procedure can be slightly modified for recovering general signals $f$ of the form (1.1). We will close our discussion by considering the case where the amplitudes $a_{i},, i = 1,...,m$ may collide and how the main procedure should probably be modified to include also this case, and give a numerical example to illustrate the main results in the text. Finding a reconstruction formula in the general case where the amplitudes may collide is left for future research.

\section{Mathematical Background}

Denote by $\Bbb R^{n}$ the standard $n$ dimensional Euclidean space, by $\Bbb S^{n - 1}$ the unit sphere in $\Bbb R^{n}$ and by $\Bbb R^{+}$ the ray $[0,\infty)$.

Denote by $C(\Bbb R^{n})$ the set of continuous real functions, defined on $\Bbb R^{n}$, with the inner product
$$\langle f,g\rangle_{\Bbb R^{n}} = \int_{\Bbb R^{n}}f(x)g(x)dx$$
in case where the integral converges. In the same way we define the set $C(\Bbb R^{+})$ with its inner product $\langle\hskip0.1cm,\hskip0.1cm \rangle_{\Bbb R^{+}}$.

For $x_{0}\in\Bbb R^{n}, \theta\in\Bbb S^{n - 1}$ and $\rho > 0$ define the following distributions on $C(\Bbb R^{n})$:

$$\delta_{x_{0}}(f) = f(x_{0}), \delta_{(\theta, \rho)}(f) = \int_{\langle x,\theta\rangle = \rho}f(x)dm_{x}$$
where $\langle\hskip0.1cm,\hskip0.1cm \rangle$ denotes the usual scalar product on $\Bbb R^{n}$.

\begin{defin}
For a given point $x\in\Bbb R^{n}$, the spherical mean transform (SMT for short) at the point $x$ is defined to be the following distribution
$$R_{x}:C\left(\Bbb R^{n}\right)\rightarrow C\left(\Bbb R^{+}\right)$$
$$R_{x}(f)(t) = t^{n - 1}\int_{|\theta| = 1}f(x + t\theta)d\theta, t\geq0.$$
\end{defin}

For a point $x\in\Bbb R^{n}$, if $f\in C\left(\Bbb R^{n}\right)$ and $\Lambda\in C(\Bbb R^{+})$ we have
$$\left\langle R_{x}f, \Lambda\right\rangle_{\Bbb R^{+}} = \int_{0}^{\infty}R_{x}(f)(t)\Lambda(t)dt = \int_{0}^{\infty}t^{n - 1}\int_{|\theta| = 1}f(x + t\theta)d\theta\Lambda(t)dt$$
in case where the last integral converges. Making the change of variables
$$y = x + t\theta, dy = t^{n - 1}d\theta dt,$$
yields
$$\left\langle R_{x}f, \Lambda\right\rangle_{\Bbb R^{+}} = \int_{\Bbb R^{n}}f(y)\Lambda(|x - y|)dy = \left\langle f, \Lambda(|x - .|) \right\rangle_{\Bbb R^{n}}.$$
Hence we define the dual SMT at $x$
$$R_{x}^{\ast}:C\left(\Bbb R^{+}\right)\rightarrow C\left(\Bbb R^{n}\right)$$
by
$$R_{x}^{\ast}\left(\Lambda\right)(y) = \Lambda(|x - y|).$$
Hence we arrive to the following definition:
\begin{defin}
Let $x\in\Bbb R^{n}$ and $T:C\left(\Bbb R^{n}\right)\rightarrow\Bbb R$ be a given distribution, then the spherical mean transform $R_{x}$ of $T$ is defined by
$$R_{x}T:C\left(\Bbb R^{+}\right)\rightarrow\Bbb R,$$
$$(R_{x}T)(\Lambda) = T(R_{x}^{\ast}\Lambda) = T(\Lambda(|x - .|)).$$
\end{defin}

\section{Main Results}

\begin{thm}
Let $m$ be a given positive integer and let $f:C(\Bbb R^{n})\rightarrow\Bbb R$ be a distribution of the form
\begin{equation}
f = \sum_{k = 1}^{m}a_{k}\delta_{x_{k}}, x_{i}\in\Bbb R^{n}, a_{i}\in\Bbb R\setminus\{0\}, 1\leq i\leq m,
\end{equation}
such that $x_{i}\neq x_{j}$ and $a_{i}\neq a_{j}$ in case where $i\neq j$. Assume that the SMT of $f$ is given at $\frac{1}{2}(n\cdot m(m - 1) + 2n + 2)$ points such that there is no hyper-plane in $\Bbb R^{n}$ which contains more than $n$ of these given points. Then the points $x_{1},...,x_{m}$ and the amplitudes $a_{1},...,a_{m}$ can be uniquely recovered.
\end{thm}

\begin{thm}
Let $m$ be a given positive integer and let $f:C(\Bbb R^{n})\rightarrow\Bbb R$ be a distribution of the form
$$f = \sum_{k = 1}^{m}a_{k}\delta_{(\theta_{k},\rho_{k})},\hskip0.1cm a_{i}\in\Bbb R\setminus\{0\}, \rho_{i}\in(0, \infty), \theta_{i}\in\Bbb S^{n - 1}, 1\leq i\leq m,$$
such that $a_{i}\neq a_{j}$ for $i\neq j$ and the hyperplanes $\langle x,\theta_{1}\rangle = \rho_{1},...,\langle x,\theta_{m}\rangle = \rho_{m}$ are all distinct (as subsets of $\Bbb R^{n}$). Assume that the SMT of $f$ is given at $n\cdot m(m - 1) + 2n + 1$ points such that there is no hyper-plane in $\Bbb R^{n}$ which contains more than $n$ of these given points. Then the parameters  $a_{1},\rho_{1},\theta_{1},...,a_{m},\rho_{m}, \theta_{m}$ can be uniquely recovered.
\end{thm}

\begin{thm}
Let $m$ be a given positive integer and $g$ be a given function, defined on $\Bbb R^{+}$, such that its radial extension belongs to the Schwartz space $\mathrm{S}\left(\Bbb R^{n}\right)$. Let $f:\Bbb R^{n}\rightarrow\Bbb R$ be a function of the form
$$f(x) = \sum_{k = 1}^{m}a_{k}g\left(\left|x - x_{k}\right|\right), x_{i}\in\Bbb R^{n}, a_{i}\in\Bbb R\setminus\{0\}, 1\leq i\leq m,$$
such that $x_{i}\neq x_{j}$ and $a_{i}\neq a_{j}$ if $i\neq j$. Assume that the SMT of $f$ is given at $\frac{1}{2}(n\cdot m(m - 1) + 2n + 2)$ points $y$ such that there is no hyper-plane in $\Bbb R^{n}$ which contains more than $n$ of these given points. Then the points $x_{1},...,x_{m}$ and the amplitudes $a_{1},...,a_{m}$ can be uniquely recovered.
\end{thm}

\section{Proofs}

\begin{proof}[\bf{Proof of Theorem 3.1:}]
Denote by $\Gamma$ the set of points on which the SMT of $f$ is given. Then, for every $y\in\Gamma$ and $h_{l}\in C(\Bbb R^{+})$, where $h_{l}(t) = t^{l}\hskip0.1cm(l\in\Bbb N\cup\{0\})$, we have
$$(R_{y}f)(h_{l}) = f(R_{y}^{\ast}h_{l}) = f\left(h_{l}\left(|y - .|\right)\right) $$
$$ = \sum_{k = 1}^{m}a_{k}\delta_{x_{k}}\left(h_{l}\left(|y - .|\right)\right) = \sum_{k = 1}^{m}a_{k}h_{l}\left(|y - x_{k}|\right) = \sum_{k = 1}^{m}a_{k}|y - x_{k}|^{l}.$$

Now, if we denote $\tau_{l} = (R_{y}f)(h_{l})$ and take $l = 0,1,...,2m - 1$ we get the following system of equations
\begin{equation}
\left(
    \begin{array}{cccc}
        1 & 1 & ... & 1\\
        |y - x_{1}| & |y - x_{2}| & ... & |y - x_{m}|\\
        ...........\\
        |y - x_{1}|^{2m - 1} & |y - x_{2}|^{2m - 1} & ... & |y - x_{m}|^{2m - 1}
    \end{array}
\right)
\left(
    \begin{array}{c}
    a_{1}\\
    a_{2}\\
    ...\\
    a_{m}
    \end{array}
\right) =
\left(
\begin{array}{c}
\tau_{0}\\
\tau_{1}\\
...\\
\tau_{2m - 1}
\end{array}
\right) = \overline{\tau}.
\end{equation}
The system of equations (4.1), where $\lambda_{i} = |y - x_{i}|$ and $a_{i}$ for $i = 1,...,m$ are the unknown variables, is of Prony's type and there is a well known literature for the solution of this type of equations (see \cite{3, 5, 8, 12, 13, 17}). However, we would like to solve this system explicitly since we want to show where the conditions that $a_{i}\neq 0, i = 1,...,m$ and that there is no hyper-plane which passes through more than $n$ points in $\Gamma$ are used. We will follow the method that was introduced in \cite{12}.

For this, let
$$p(t) = c_{0} + c_{1}t + ... + c_{m - 1}t^{m - 1} + t^{m}$$
be the unique monic polynomial whose roots (with possible multiplicities) are $t = \lambda_{1},...,\lambda_{m}$. Now observe that for $k = 0,...,m - 1$ the polynomial $t^{k}p(t)$ also vanishes at $t = \lambda_{1},...,\lambda_{m}$. Hence for every $k = 0,...,m - 1$ the vector
$$v_{k} = (\underset{k}{\underbrace{0,...,0}}, c_{0}, c_{1},...,c_{m - 1}, 1, \underset{m - 1 - k}{\underbrace{0,...,0}})$$
is in the left null space of the $2m\times m$ matrix in the left hand side of equation (4.1) and thus it is also orthogonal to the vector $\overline{\tau}$ in the right hand side. Observe that we can write the following system of equations
$$v_{k}\cdot\overline{\tau} = 0, k = 0,..., m - 1$$
as follows
\begin{equation}
\left(
    \begin{array}{cccc}
        \tau_{0} & \tau_{1} & ... & \tau_{m - 1}\\
        \tau_{1} & \tau_{2} & ... & \tau_{m}\\
        ..........\\
        \tau_{m - 1} & \tau_{m} & ... & \tau_{2m - 2}
    \end{array}
\right)
\left(
    \begin{array}{c}
        c_{0}\\
        c_{1}\\
        ...\\
        c_{m - 1}
    \end{array}
\right)
 =
- \left(
    \begin{array}{c}
        \tau_{m}\\
        \tau_{m + 1}\\
         ...\\
         \tau_{2m - 1}
    \end{array}
\right).
\end{equation}
Denote the matrix in the left hand side of (4.2) by $U$. Then, in order to solve the system of equations (4.2) we need to guarantee that $U$ is non degenerate. For this we use the following factorization $U = V\Lambda V^{T}$, where
$$ V = \left(
    \begin{array}{cccc}
        1 & 1 & ... & 1\\
        \lambda_{1} & \lambda_{2} & ... & \lambda_{m}\\
        \lambda_{1}^{2} & \lambda_{2}^{2} & ... & \lambda_{m}^{2}\\
        ...........\\
        \lambda_{1}^{m - 1} & \lambda_{2}^{m - 1} & ... & \lambda_{m}^{m - 1}
    \end{array}
\right),
\Lambda = \left(
    \begin{array}{cccc}
        a_{1} & 0 & ... & 0\\
        0 & a_{2} & ... & 0\\
        .....\\
        0 & 0 & ... & a_{m}
    \end{array}
\right),$$

which can be proved by direct calculations. Since, by assumption, $a_{i}\neq0, i = 1,...,m$ it follows that $\Lambda$ is non degenerate. Also, since $V$ is a Vandermonde matrix it follows that $V$ is degenerate if and only if there are two different indices $1\leq i,j\leq m$ such that $\lambda_{i} = \lambda_{j}$, or $|y - x_{i}| = |y - x_{j}|$. This occurs if $y$ is in equal distances from $x_{i}$ and $x_{j}$. However, we claim that we can find $n + 1$ points $y_{1},...,y_{n + 1}$ in the set $\Gamma$ such that for every $1\leq l\leq n + 1$ the following conditions hold:

\begin{equation}
|y_{l} - x_{i}| \neq |y_{l} - x_{j}|, \forall i,j, 1\leq i,j\leq m, i\neq j.
\end{equation}

Indeed, a point $y$ in $\Gamma$ satisfies $|y_{l} - x_{i}| = |y_{l} - x_{j}|$ for two different indices $i$ and $j$  if and only if it lies in the unique hyperplane $H_{i,j}$ which divides into two equal parts and is orthogonal to the vector $x_{i} - x_{j}$. Since, by assumption, there are at most $n$ points in $\Gamma$ on such hyperplane, and since there are at most $\frac{1}{2}m(m - 1)$ such hyperplanes, it follows, since $\Gamma$ contains $\frac{1}{2}(n\cdot m(m - 1) + 2n + 2)$ distinct points, that there are at least $n + 1$ points in $\Gamma$ which non of them lies in any of these hyperplanes. Hence, these points satisfy condition (4.3).

In order to check whether a point $y\in\Gamma$ satisfies condition (4.3) we just need to check whether the matrix $U$ (which depends on the point $y$) is non degenerate. As was just explained, we can find at least $n + 1$ such points.

Let $\Gamma' = \{y_{1},y_{2},...,y_{n + 1}\}$ be a subset of $\Gamma$ such that each $y_{i}\in \Gamma'$ satisfies condition (4.3). Let us take the point $y_{1}\in\Gamma'$ and build the system of equations (4.2) for this point. Since this system is non degenerate we can extract the coefficients $c_{0},c_{1},...,c_{m - 1}$ and find the roots of the polynomial $P = P(t)$ with these set of coefficients, let $\xi_{1},...,\xi_{m}$ be its roots. Now, we can assume without loss of generality that the point $x_{i}$ corresponds to $\xi_{i}$, i.e., $\xi_{i} = |y_{1} - x_{i}|$. Indeed, since there is a permutation $\sigma$ of $1,...,m$ such that $|y_{1} - x_{\sigma(i)}| = \xi_{i}$, we can just rename the points $x_{1},...,x_{m}$ and their corresponding amplitudes $a_{1},...,a_{m}$ according to this permutation. Since this renaming does not change the sum (3.1) which defines the distribution $f$, this step is valid.

Returning to equation (4.1) with $y = y_{1}$, observe that the $m\times m$ top submatrix in the left hand side of this equation is non degenrate since $|y_{1} - x_{i}| = \xi_{i}$ and the roots $\xi_{1},...,\xi_{m}$ are distinct. Hence, at this stage we can extract the amplitudes $a_{1},...,a_{m}$.

Now, our aim is to use the remaining points $y_{2},...,y_{n + 1}$ in order to extract the points $x_{1},...,x_{m}$. For this, for each $2\leq l\leq n + 1$ we solve equation (4.2) and extract the coefficients of the polynomial whose roots are

$$\{\xi_{1,l},...,\xi_{m,l}\} = \{|y_{l} - x_{1}|,...,|y_{l} - x_{m}|\}.$$

However, at this point we do not know which root $\xi_{j,l}$ corresponds to each point $x_{i}$ for $l\geq 2$. We can overcome this problem by inserting all the possible permutations of the roots $\xi_{1,l},...,\xi_{m,l}$ and check which one of them solves equation (4.1) (observe that at this point the amplitudes $a_{1},...,a_{m}$ are known). However, now the problem is that maybe there are two different permutations of these roots which solve equation (4.1). At this point we use the condition that $a_{i}\neq a_{j}, i\neq j$ to show that this case is impossible.

Indeed, suppose that there is a permutation of $\xi_{1,l},...,\xi_{m,l}$  which solves (4.1). By reordering and renaming the roots $\xi_{1,l},...,\xi_{m,l}$ we can assume that this permutation is the identity. Now suppose that there is another permutation $\sigma$ from $\{1,...,m\}$ to itself, which is different from the identity such that the permutation $\xi_{1}' = \xi_{\sigma(1),l},...,\xi_{m}' = \xi_{\sigma(m),l}$ also solves equation (4.1). Then, subtracting these two equations we have

$$
\left(
    \begin{array}{cccc}
        0 & 0 & ... &0\\
        \xi_{1,l} - \xi_{\sigma(1),l} & \xi_{2,l} - \xi_{\sigma(2),l} & ... &\xi_{m,l} - \xi_{\sigma(m),l}\\
        \xi_{1,l}^{2} - \xi_{\sigma(1),l}^{2} & \xi_{2,l}^{2} - \xi_{\sigma(2),l}^{2} & ... & \xi_{m,l}^{2} - \xi_{\sigma(m),l}^{2}\\
        ......\\
         \xi_{1,l}^{2m - 1} - \xi_{\sigma(1),l}^{2m - 1} & \xi_{2,l}^{2m - 1} - \xi_{\sigma(2),l}^{2m - 1} & ... & \xi_{m,l}^{2m - 1} - \xi_{\sigma(m),l}^{2m - 1}\\
    \end{array}
\right)
\left(
    \begin{array}{c}
        a_{1}\\
        a_{2}\\
        ...\\
        a_{m}
    \end{array}
\right)
 = \left(
    \begin{array}{c}
    0\\
    0\\
    ...\\
    0
    \end{array}
 \right).
$$

This in particular implies that

\begin{equation}
\left(
    \begin{array}{cccc}
        \xi_{1,l} - \xi_{\sigma(1),l} & \xi_{2,l} - \xi_{\sigma(2),l} & ... &\xi_{m,l} - \xi_{\sigma(m),l}\\
        \xi_{1,l}^{2} - \xi_{\sigma(1),l}^{2} & \xi_{2,l}^{2} - \xi_{\sigma(2),l}^{2} & ... &\xi_{m,l}^{2} - \xi_{\sigma(m),l}^{2}\\
        ......\\
         \xi_{1,l}^{m} - \xi_{\sigma(1),l}^{m} & \xi_{2,l}^{m} - \xi_{\sigma(2),l}^{m} & ... &\xi_{m,l}^{m} - \xi_{\sigma(m),l}^{m}\\
    \end{array}
\right)
\left(
    \begin{array}{c}
        a_{1}\\
        a_{2}\\
        ...\\
        a_{m}
    \end{array}
\right)
 = \left(
    \begin{array}{c}
    0\\
    0\\
    ...\\
    0
    \end{array}
 \right).
\end{equation}

\vskip0.3cm

Now we can use Lemma 5.2. Indeed, from (4.3) it follows that $\xi_{i,l}\neq\xi_{j,l}$ if $i\neq j$ and since $\sigma$ is different from the identity it follows that all the conditions of Lemma 5.2 are satisfied. Hence, any vector in the kernel of the matrix in the left hand side of (4.4) must have two equal components with different indices. This implies that $a_{i} = a_{j}$ for $i\neq j$ which is a contradiction to our initial assumption.

Hence, for each $1\leq i\leq m$ we know the following distances:

$$\mu_{i,l} = |y_{l} - x_{i}|, l = 1,...,n + 1.$$

Now since the there is no hyperplane in $\Bbb R^{n}$ which contains the points $y_{1},...,y_{n + 1}$, it is straightforward to check that from the data $\mu_{i,1},...,\mu_{i,n + 1}$ we can uniquely recover the point $x_{i}$.

\end{proof}

\begin{proof}[\bf{Proof of Theorem 3.2:}]

We will follow the same ideas as in the proof of Theorem 3.1. Denote by $\Gamma$ the set on which the SMT of $f$ is given. For every $y\in \Gamma$ and $h_{l}\in C(\Bbb R^{+})$, where $h_{l}(t) = e^{-lt^{2}}$ ($l\in\Bbb N$), we have
$$(R_{y}f)(h_{l}) = f(R_{y}^{\ast}h_{l}) = f(h_{l}(|y - .|)) = \sum_{k = 1}^{m}a_{k}\delta_{(\theta_{k},\rho_{k})}(h_{l}(|y - .|))$$
\begin{equation}
 = \sum_{k = 1}^{m}a_{k}\int_{\langle x, \theta_{k}\rangle = \rho_{k}}h_{l}(|y - x|)dm_{x} = \sum_{k = 1}^{m}a_{k}\int_{\langle x, \theta_{k}\rangle = \rho_{k}}e^{-l|y - x|^{2}}dm_{x}.
 \end{equation}

For each integral in the sum in the right hand side of equation (4.5), we make the following change of variables $x = y + A_{k}^{T}z, dm_{x} = dm_{z}$ where $A_{k}$ is an orthogonal matrix which satisfies $A_{k}\theta_{k} = e_{n}$ (observe that we do not know the matrix $A_{k}$, but we can still formally make this change of variables). Hence from equation (4.5) we have

$$ (R_{y}f)(h_{l}) = \sum_{k = 1}^{m}a_{k}\int_{z_{n} = \rho_{k} - \langle y, \theta_{k}\rangle}e^{-l|z|^{2}}dm_{z}$$
$$ = \sum_{k = 1}^{m}a_{k}e^{-l(\rho_{k} - \langle y, \theta_{k}\rangle)^{2}}\int_{\Bbb R^{n- 1}}e^{-l\left(z_{1}^{2} + ... + z_{n - 1}^{2}\right)}dz_{1}...dz_{n - 1}$$
$$ = \left(\frac{\pi}{l}\right)^{\frac{n - 1}{2}}\sum_{k = 1}^{m}a_{k}e^{-l|\rho_{k} - \langle y, \theta_{k}\rangle|^{2}}.$$
Now if we denote
$$\tau_{l} = \left(\frac{\pi}{l}\right)^{-\frac{n - 1}{2}}(R_{y}f)(h_{l}), \lambda_{k} = e^{-|\rho_{k} - \langle y, \theta_{k}\rangle|^{2}}$$
and take $l = 1,...,2m$ we can arrive, exactly as in Theorem 3.1, to the following system of equations:

\begin{equation}
\left(
    \begin{array}{cccc}
        \tau_{1} & \tau_{2} & ... & \tau_{m}\\
        \tau_{2} & \tau_{3} & ... & \tau_{m + 1}\\
        ..........\\
        \tau_{m} & \tau_{m + 1} & ... & \tau_{2m - 1}
    \end{array}
\right)
\left(
    \begin{array}{c}
        c_{0}\\
        c_{1}\\
        ...\\
        c_{m - 1}
    \end{array}
\right)
 =
- \left(
    \begin{array}{c}
        \tau_{m + 1}\\
        \tau_{m + 2}\\
         ...\\
         \tau_{2m}
    \end{array}
\right),
\end{equation}

\vskip0.2cm

where $c_{0},...,c_{m - 1}$ are the coefficients of the unique monic polynomial whose roots (with possible multiplicities) are $\lambda_{1},...,\lambda_{m}$. Denote the matrix in the left hand side of (4.6) by $U$, then we have the following factorization $U = V_{1}\Lambda V_{2}$, where
$$ V_{1} = \left(
    \begin{array}{cccc}
        \lambda_{1} & \lambda_{2} & ... & \lambda_{m}\\
        \lambda_{1}^{2} & \lambda_{2}^{2} & ... & \lambda_{m}^{2}\\
        ......\\
        \lambda_{1}^{m} & \lambda_{2}^{m} & ... & \lambda_{m}^{m}
    \end{array}
\right),
V_{2} = \left(
    \begin{array}{cccc}
        1 & \lambda_{1} & ... & \lambda_{1}^{m - 1}\\
        1 & \lambda_{2} & ... & \lambda_{2}^{m - 1}\\
        ......\\
        1 & \lambda_{m} & ... & \lambda_{m}^{m - 1}
    \end{array}
\right),
\Lambda = \left(
    \begin{array}{cccc}
        a_{1} & 0 & ... & 0\\
        0 & a_{2} & ... & 0\\
        .....\\
        0 & 0 & ... & a_{m}
    \end{array}
\right)$$

which can be proved by direct calculations. Since both $V_{1}$ and $V_{2}$ are Vandermonde's type matrices and $\lambda_{i}\neq0, i = 1,...,m$, it follows that $V_{1}$ or $V_{2}$ are degenerate if and only if there exist two indices $i\neq j$ such that $\lambda_{i} = \lambda_{j}$. This occurs only if
$$|\rho_{i} - \langle y, \theta_{i}\rangle| = |\rho_{j} - \langle y, \theta_{j}\rangle| $$
which implies that $y$ is at the same distance from the planes $\langle x, \theta_{i}\rangle = \rho_{i} $, $\langle x, \theta_{j}\rangle = \rho_{j}$, $x\in \Bbb R^{n}$. We claim that we can find $2n + 1$ points $y_{1},...,y_{2n + 1}$ in the set $\Gamma$ such that for every $1 \leq l \leq 2n + 1$ the following conditions hold:

\begin{equation}
|\rho_{i} - \langle y_{l}, \theta_{i}\rangle| \neq |\rho_{j} - \langle y_{l}, \theta_{j}\rangle|, \forall i,j, 1\leq i,j\leq m, i\neq j.
\end{equation}

Indeed, a point $y$ in $\Gamma$ satisfies $|\rho_{i} - \langle y, \theta_{i}\rangle| = |\rho_{j} - \langle y, \theta_{j}\rangle|$ for two different indices $i$ and $j$ if and only if it lies on one of the hyperplanes
$$\langle x, \theta_{i} + \theta_{j}\rangle = \rho_{i} + \rho_{j}, \langle x, \theta_{i} - \theta_{j}\rangle = \rho_{i} - \rho_{j}$$

(in case where $\theta_{j} = \pm\theta_{i}$ then this occurs only if $y$ lies on the hyperplane\newline $\langle x,\theta_{i}\rangle = \frac{1}{2}\left(\rho_{i} \pm \rho_{j}\right)$). Since, by assumption, there are at most $n$ points in $\Gamma$ on any of these hyperplanes, and since there are at most $m(m - 1)$ such hyperplanes, it follows, since $\Gamma$ contains $n\cdot m(m - 1) + 2n + 1$ distinct points, that there are at least $2n + 1$ points in $\Gamma$ which non of them lies in any of these hyperplanes. Hence, these points satisfy condition (4.7).

In order to check whether a point $y \in \Gamma$ satisfies condition (4.7) we just need to check whether the matrix $U$ (which depends on the point $y$) is non degenerate. As was just explained, we can find at least $2n + 1$ such points.

Now, exactly as we did in Theorem 3.1, we extract a subset $\Gamma' = \{y_{1},...,y_{2n + 1}\}$ of $\Gamma$ such that each point $y_{l}$ in $\Gamma'$ satisfies condition (4.7). We take the point $y_{1}$ and build its corresponding polynomial whose roots are
$$\{\xi_{1},...,\xi_{m}\} = \{|\rho_{1} - \left\langle y_{1}, \theta_{1}\right\rangle|,...,|\rho_{m} - \left\langle y_{1}, \theta_{m}\right\rangle|\}.$$

As we did in Theorem 3.1, we can assume without loss of generality that the hyperplane $\rho_{i} - \left\langle y, \theta_{i}\right\rangle$ corresponds to the root $\xi_{i}$ and then solve a system of equations similar to (4.1) in order to determine the amplitudes $a_{1},...,a_{m}$. \\

In order to extract the parameters $\rho_{1},\theta_{1},..., \rho_{m},\theta_{m}$ we use the remaining points $y_{2},...,y_{2n + 1}$ in $\Gamma'$. For each $2\leq l\leq 2n + 1$ we build the monic polynomial whose roots are
$$\{\xi_{1,l},...,\xi_{m,l}\} = \{|\rho_{1} - \left\langle y_{l}, \theta_{1}\right\rangle|,...,|\rho_{m} - \left\langle y_{l}, \theta_{m}\right\rangle|\}$$

and exactly as we did in Theorem 3.1, we can find the order in which these roots correspond to the $m$ pairs $(\theta_{1}, \rho_{1}),...,(\theta_{m}, \rho_{m})$. Hence, for every $y_{l}\in \Gamma', 1\leq l\leq 2n + 1$ we know the following distances:
$$|\rho_{1} - \langle y_{l},\theta_{1}\rangle|,...,|\rho_{m} - \langle y_{l},\theta_{m}\rangle|.$$

Hence for every $1 \leq i \leq m$ we have the following data:
\begin{equation}
|\rho_{i} - \langle y_{l},\theta_{i}\rangle|, l = 1,...,2n + 1.
\end{equation}
Now we claim that the data given in (4.8) determines the hyperplane $H : \langle x, \theta_{i}\rangle = \rho_{i}$ uniquely. Indeed, suppose that there is another hyperplane $H'$ such that every point $y_{l}, l = 1,...,2n + 1$ has equal distances from $H$ and $H'$. The set of points which have equal distances from $H$ and $H'$ is the union of two hyperplanes. Hence the points $y_{l}, l = 1,...,2n + 1$ lie in these hyperplanes and thus it follows that at least one hyperplane contains $n + 1$ of these points which is a contradiction to our assumption on the points in the set $\Gamma$. Hence Theorem 3.2 is proved.

\end{proof}

\begin{proof}[\bf{Proof of Theorem 3.3:}]
Let $G$ be the Hankel transform of $g$ of order $\frac{n}{2} - 1$:
\begin{equation}
G(\lambda) = \int_{0}^{\infty}g(r)j_{\frac{n}{2} - 1}(\lambda r)r^{n - 1}dr, \lambda\geq 0,
\end{equation}
where $j_{\nu}$ is the normalised Bessel function of order $\nu$, $j_{\nu}(\lambda) = \lambda^{-\nu}J_{\nu}(\lambda)$. Since the radial extension of $g$ is in $\mathrm{S}\left(\Bbb R^{n}\right)$, it follows from Lemma 5.1 that $g$ is continuously differentiable and belongs to $L^{1}\left(\Bbb R^{+}, r^{\frac{n - 1}{2}}\right)$. Hence, the Hankel transform (4.9) can be inverted in the following way:
$$g(r) = \int_{0}^{\infty}G(\lambda)j_{\frac{n}{2} - 1}(\lambda r)\lambda^{n - 1}d\lambda, r\geq0.$$
Hence, from the definition of the function $f$ it follows that
$$f(x) = \sum_{k = 1}^{m}a_{k}\int_{0}^{\infty}G(\lambda)j_{\frac{n}{2} - 1}\left(\lambda\left|x - x_{k}\right|\right)\lambda^{n - 1}d\lambda.$$
Thus, taking the SMT of $f$ at a point $y$ we have
$$(R_{y}f)(t) = \sum_{k = 1}^{m}a_{k}\int_{0}^{\infty}G(\lambda)\left(R_{y}\left(j_{\frac{n}{2} - 1}\left(\lambda\left|. - x_{k}\right|\right)\right)\right)(t)\lambda^{n - 1}d\lambda$$
$$ = \sum_{k = 1}^{m}a_{k}\int_{0}^{\infty}G(\lambda)\int_{|\theta| = 1}j_{\frac{n}{2} - 1}\left(\lambda\left|y - x_{k} + t\theta\right|\right)d\theta\lambda^{n - 1}d\lambda t^{n - 1}.$$
Using the following indentity
$$\int_{|\theta| = 1}j_{\frac{n}{2} - 1}\left(\lambda\left|x + r\theta\right|\right)d\theta = (2\pi)^{\frac{n}{2}}j_{\frac{n}{2} - 1}(\lambda r)j_{\frac{n}{2} - 1}(\lambda|x|)$$
on $R_{y}f$ yields
$$(R_{y}f)(t) = (2\pi)^{\frac{n}{2}}\sum_{k = 1}^{m}a_{k}\int_{0}^{\infty}G(\lambda)j_{\frac{n}{2} - 1}(\lambda|y - x_{k}|)j_{\frac{n}{2} - 1}(\lambda t)\lambda^{n - 1}d\lambda t^{n - 1}.$$
Hence from the last equation we have
\begin{equation}
\frac{t^{1 - n}(R_{y}f)(t)}{(2\pi)^{\frac{n}{2}}} = \int_{0}^{\infty}F(\lambda)j_{\frac{n}{2} - 1}(\lambda t)\lambda^{n - 1}d\lambda
\end{equation}
where
$$F(\lambda) = \left(\sum_{k = 1}^{m}a_{k}j_{\frac{n}{2} - 1}(\lambda|y - x_{k}|)\right)G(\lambda).$$
Now observe that the right hand side of equation (4.10) is the Hankel transform of $F$. In order to use the inverse Hankel transform on (4.10) we need to make sure that $F$ is continuously differentiable and belongs to $L^{1}\left(\Bbb R^{+}, r^{\frac{n - 1}{2}}\right)$.

From the definition of $F$ it follows that if $G$ belongs to $L^{1}\left(\Bbb R^{+}, r^{\frac{n - 1}{2}}\right)$ and is continuously differentiable then the same is true for $F$, these facts on $G$ are proved in Lemma 5.1. Hence, we can take the inverse Hankel transform on equation (4.10):
\begin{equation}
\frac{1}{(2\pi)^{\frac{n}{2}}}\int_{0}^{\infty}(R_{y}f)(t)j_{\frac{n}{2} - 1}(\lambda t)dt = \left(\sum_{k = 1}^{m}a_{k}j_{\frac{n}{2} - 1}(\lambda|y - x_{k}|)\right)G(\lambda).
\end{equation}
Dividing equation (4.11) by $G(\lambda)$ and then taking the derivatives $2k$ times with respect to $\lambda$ at $\lambda = 0$ yields
\begin{equation}
\frac{2^{2k - 1}(-1)^{k}k!\Gamma\left(\frac{n}{2} + k\right)}{\pi^{\frac{n}{2}}(2k)!}\left.\left(\frac{1}{G(\lambda)}
\int_{0}^{\infty}(R_{y}f)(t)j_{\frac{n}{2} - 1}(\lambda t)dt\right)^{(2k)}\right|_{\lambda = 0} = \sum_{k = 1}^{m}a_{k}\left|y - x_{k}\right|^{2k}.
\end{equation}
Now we can repeat the same procedure we used in Theorem 3.1. That is, by our assumption, the SMT of $f$ is given at $\frac{1}{2}(n\cdot m(m - 1) + 2n + 2)$ points $y$ such that there is no hyper-plane in $\Bbb R^{n}$ which contains more than $n$ of these given points. Hence, replacing each of these points $y$ in equation (4.12) and letting $k = 0,1,...,2m - 1$ yields a system of equations from which the points $x_{1},...,x_{m}$ and the amplitudes $a_{1},...,a_{m}$ can be uniquely recovered.

\end{proof}

\subsection{A Remark for the Case Where the Amplitudes Collide}

If the distribution $f$ has the form (3.1) and some of its amplitudes have the same value, then the procedure we used in Theorem 3.1, in order to extract $f$, does not work anymore. The reason is that the proof of Theorem 3.1 uses only $n + 1$ points which we know a priori to have different distances from each of the nodes $x_{i}, i = 1,...,m$. However, in case when the amplitudes collide we cannot determine, after extracting the roots of the corresponding polynomial, which root corresponds to each node $x_{i}$. This lack of information is crucial since without it the distribution $f$ cannot be uniquely determined.

For example, for $n = m = 2$ let us take the following two distributions
$$f_{1} = \delta(x - p_{1}) + \delta(x - p_{2}), f_{2} = \delta(x - q_{1}) + \delta(x - q_{2}),$$
where
$$p_{1} = (0,1), p_{2} = (2,-1), q_{1} = (0,-1), q_{2} = (2,1).$$
Since the proof of Theorem 3.1 uses only $n + 1 = 3$ points $y_{1},y_{2}$ and $y_{3}$, not on the same line, such that each point is known a priori to have different distances from each node that we want to extract, we can assume that the SMT is given only at these points. However, in this case the information received from the SMT can collide for the distributions $f_{1}$ and $f_{2}$. Indeed, if we take
$$y_{1} = (0,0), y_{2} = (2,0), y_{3} = (1,1)$$
then the points $y_{1}, y_{2}$ and $y_{3}$ are not on the same line and each of them has different distances from the nodes $p_{1}$ and $p_{2}$ of $f_{1}$ and the same is true for the distribution $f_{2}$. However, since
$$|y_{i} - p_{1}| = |y_{i} - q_{1}|, |y_{i} - p_{2}| = |y_{i} - q_{2}|, i = 1,2,$$
$$|y_{3} - p_{1}| = |y_{3} - q_{2}|, |y_{3} - p_{2}| = |y_{3} - q_{1}|,$$
the SMT receives the same information from $f_{1}$ and $f_{2}$ and hence they cannot be distinguished (see the picture below).

\vskip0.5cm

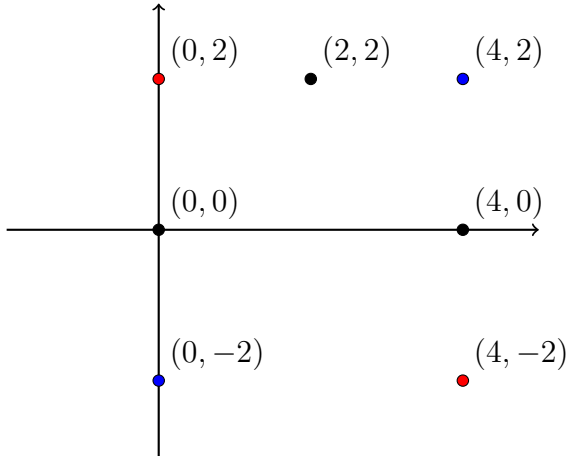
\begin{tikzpicture}

\draw[thick,->] (-2,0) -- (5,0);
\draw[thick,->] (0,-3) -- (0,3);
\foreach \Point/\PointLabel in {(0,0)/y_1, (4,0)/y_2, (2,2)/y_3}
\draw[fill=black] \Point circle (0.075) node[above right] {$\PointLabel$};
\foreach \Point/\PointLabel in {(0,2)/p_1, (4,-2)/p_2}
\draw[fill=red] \Point circle (0.075) node[above right] {$\PointLabel$};
\foreach \Point/\PointLabel in {(0,-2)/q_1, (4,2)/q_2}
\draw[fill=blue] \Point circle (0.075) node[above right] {$\PointLabel$};

\end{tikzpicture}

The same problem also occurs in the case of hyperplanes. For example, for $n = m = 2$ let us take the following two distributions
$$f_{1} = \delta_{l_{1}} + \delta_{l_{2}}, f_{2} = \delta_{k_{1}} + \delta_{k_{2}},$$
where
$$l_{1}: x - 2y = 0, l_{2}: 2x + y = 0, k_{1}: x + 2y = 0, k_{2}: 2x - y = 0.$$

The proof of Theorem 3.2 uses in this case only $5$ points such that no line passes through $3$ of them and such that each point is known a priori to have different distances from the lines $l_{1}$ and $l_{2}$ (and the same is true for the lines $k_{1}$ and $k_{2}$). However, in this case the information received from the
SMT can collide for the distributions $f_{1}$ and $f_{2}$. Indeed, if we take
$$y_{1} = (-1,0), y_{2} = (1,0), y_{3} = (0,-1), y_{4} = (0,1), y_{5} = (1,1),$$
then no line contains more than $3$ of these points and each of these points is at different distances from the lines $l_{1}$ and $l_{2}$ (the same is true for $k_{1}$ and $k_{2}$). However, since
$$d(y_{i}, l_{1}) = d(y_{i}, k_{1}), d(y_{i}, l_{2}) = d(y_{i}, k_{2}), i = 1,2,3,4,$$
$$d(y_{5}, l_{1}) = d(y_{5}, k_{2}), d(y_{5}, l_{2}) = d(y_{5}, k_{1}),$$
the SMT receives the same information from $f_{1}$ and $f_{2}$ and hence they cannot be distinguished (see the picture below).

\vskip0.5cm

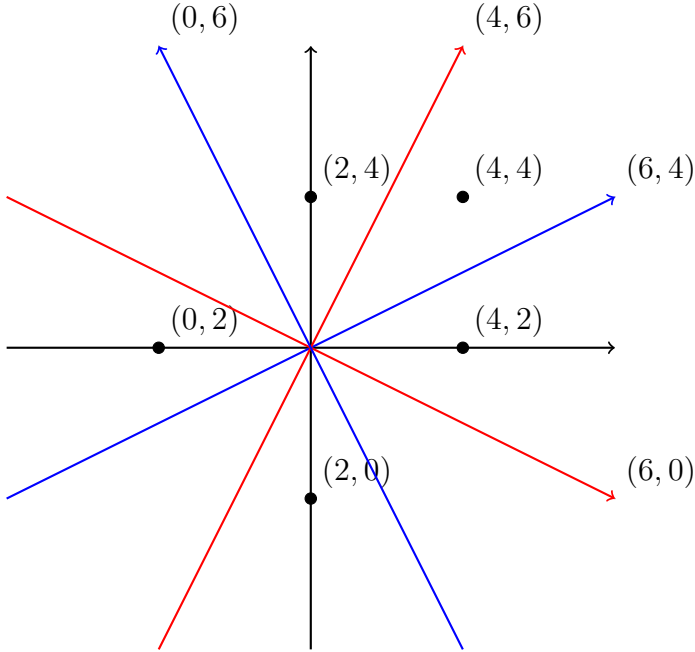
\begin{tikzpicture}

\draw[thick,->] (-2,2) -- (6,2);
\draw[thick,->] (2,-2) -- (2,6);

\foreach \Point/\PointLabel in {(0,2)/y_1, (4,2)/y_2, (2,0)/y_3, (2,4)/y_4, (4,4)/y_5}
\draw[fill=black] \Point circle (0.075) node[above right] {$\PointLabel$};

\foreach \Point/\PointLabel in {(6,4)/l_1, (0,6)/l_2, (6,0)/k_1, (4,6)/k_2}
\draw[fill=black] \Point circle (0.001) node[above right] {$\PointLabel$};

\draw[blue][thick,->] (-2,0) -- (6,4);
\draw[red][thick,->] (-2,4) -- (6,0);

\draw[red][thick,->] (0,-2) -- (4,6);
\draw[blue][thick,->] (4,-2) -- (0,6);

\end{tikzpicture}

\vskip0.5cm

The author of this article believes that Theorems 3.1 and 3.2 remain true in the general case where the amplitudes may collide. However, in order to find a procedure to extract the nodes in Theorem 3.1 (or the hyperplanes in Theorem 3.2) one will also have to use the information received from the points, on which the SMT is given, whose distances from some of the nodes $x_{i}, 1\leq i\leq m$ may collide.

\subsection{A Numerical Example}

We finish by giving a numerical example for extracting distributions of the form (3.1) in case where $n = m = 2$. Suppose that
$$f = 3\delta_{x_{1}} + 2\delta_{x_{2}}, \textrm{where}\hskip0.1cm x_{1} = (-1, 0), x_{2} = (1, 0)$$
and that the SMT of $f$ is given at the points
$$y_{1} = (0,0), y_{2} = (0,2), y_{3} = (-1, 1), y_{4} = (1,1), y_{5} = (1,2)$$
(observe that no line passes through more than two of these points). Computing the values $\tau_{i}, i = 0,1,2,3$ for the point $y_{1}$ we obtain $(\tau_{0}, \tau_{1}, \tau_{2}, \tau_{3})_{y_{1}} = (5,5,5,5)$. Similarly, for the rest of the points we have
$$(\tau_{0}, \tau_{1}, \tau_{2}, \tau_{3})_{y_{2}} = (5, 11.18, 25, 55.901),
(\tau_{0}, \tau_{1}, \tau_{2}, \tau_{3})_{y_{3}} = (5, 7.472, 13, 25.36),$$
$$(\tau_{0}, \tau_{1}, \tau_{2}, \tau_{3})_{y_{4}} = (5, 8.708, 17, 35.541),
(\tau_{0}, \tau_{1}, \tau_{2}, \tau_{3})_{y_{5}} = (5, 12.485, 32, 75.882).$$
Building the corresponding equations for the polynomials coefficients we obtain
$$\hskip-9cm\left(\begin{array}{cc}5 & 5 \\ 5 & 5\end{array}\right)\left(\begin{array}{c}c_{0} \\ c_{1}\end{array}\right) = -\left(\begin{array}{c}5 \\ 5\end{array}\right),$$
$$\hskip-6.58cm\left(\begin{array}{cc}5 & 11.18 \\ 11.18 & 25\end{array}\right)\left(\begin{array}{c}c_{0} \\ c_{1}\end{array}\right) = -\left(\begin{array}{c}25 \\ 55.901\end{array}\right),$$
$$\hskip-6.78cm\left(\begin{array}{cc}5 & 7.472 \\ 7.472 & 13\end{array}\right)\left(\begin{array}{c}c_{0} \\ c_{1}\end{array}\right) = -\left(\begin{array}{c}13 \\ 25.36\end{array}\right),$$
$$\hskip-6.58cm\left(\begin{array}{cc}5 & 8.708 \\ 8.708 & 17\end{array}\right)\left(\begin{array}{c}c_{0} \\ c_{1}\end{array}\right) = -\left(\begin{array}{c}17 \\ 35.541\end{array}\right),$$
$$\hskip-6.16cm\left(\begin{array}{cc}5 & 12.485 \\ 12.485 & 32\end{array}\right)\left(\begin{array}{c}c_{0} \\ c_{1}\end{array}\right) = -\left(\begin{array}{c}32 \\ 75.882\end{array}\right).$$
\vskip0.3cm
The first two systems of equations are degenerate as expected since the points $y_{1}$ and $y_{2}$ have equal distances from the points $x_{1}$ and $x_{2}$. The other three systems have a unique solution and thus we obtain the following three corresponding polynomials
$$\hskip-8.5cm P_{y_{3}}(x) = 2.234 - 3.235x + x^{2},$$
$$\hskip-8.5cm P_{y_{4}}(x) = 2.234 - 3.235x + x^{2},$$
\vskip-0.3cm
$$\hskip-8.5cm P_{y_{5}}(x) = 5.656 - 4.828x + x^{2},$$
whose roots are
$$\hskip-3cm \xi_{1,3} = 1, \xi_{2,3} = 2.235, \xi_{1,4} = 1, \xi_{2,4} = 2.235, \xi_{1,5} = 2, \xi_{2,5} = 2.828.$$

Now we can assume without loss of generality that the point $x_{1}$ corresponds to the root $\xi_{1,3}$ and that $x_{2}$ corresponds to $\xi_{2,3}$. Hence, taking the first two equations of the system (4.1) we obtain
$$\left(\begin{array}{cc} 1 & 1 \\ 1 & 2.235\end{array}\right)\left(\begin{array}{c} a_{1} \\ a_{2} \end{array}\right) = \left(\begin{array}{c} 5 \\ 7.472 \end{array}\right).$$
The solution of this system is $a_{1} = 2.998, a_{2} = 2.001$. Hence at this point the amplitudes are (approximately) known. In order to decide how the points $x_{1}$ and $x_{2}$ correspond to the roots $\xi_{1,4} = 1, \xi_{2,4} = 2.235$ we need to check which one of the following systems of equations is valid
$$\left(\begin{array}{cc} 1 & 1 \\ \xi_{1,4} & \xi_{2,4}\end{array}\right)\left(\begin{array}{c} 2.998 \\ 2.001 \end{array}\right) = \left(\begin{array}{c} 5 \\ 8.708 \end{array}\right),
\left(\begin{array}{cc} 1 & 1 \\ \xi_{2,4} & \xi_{1,4}\end{array}\right)\left(\begin{array}{c} 2.998 \\ 2.001 \end{array}\right) = \left(\begin{array}{c} 5 \\ 8.708 \end{array}\right).$$
Since only the second system of equations is valid it follows that $x_{1}$ corresponds to $\xi_{2,4}$ and $x_{2}$ corresponds to $\xi_{1,4}$. Finally, in order to decide how the points $x_{1}$ and $x_{2}$ correspond to the roots $\xi_{1,5} = 2, \xi_{2,5} = 2.828$ we need to check which one of the following systems of equations is valid
$$\left(\begin{array}{cc} 1 & 1 \\ \xi_{1,5} & \xi_{2,5}\end{array}\right)\left(\begin{array}{c} 2.998 \\ 2.001 \end{array}\right) = \left(\begin{array}{c} 5 \\ 12.485 \end{array}\right),
\left(\begin{array}{cc} 1 & 1 \\ \xi_{2,5} & \xi_{1,5}\end{array}\right)\left(\begin{array}{c} 2.998 \\ 2.001 \end{array}\right) = \left(\begin{array}{c} 5 \\ 12.485 \end{array}\right).$$
Since only the second system of equations is valid it follows that $x_{1}$ corresponds to $\xi_{2,5}$ and $x_{2}$ corresponds to $\xi_{1,5}$.

Since $x_{1}$ corresponds to $\xi_{1,3}$, $\xi_{2,4}$ and $\xi_{2,5}$ we have the following system of equations for $x_{1}$
$$(x_{1,1} - y_{3,1})^{2} + (x_{1,2} - y_{3,2})^{2} = \xi_{1,3}^2,$$
$$(x_{1,1} - y_{4,1})^{2} + (x_{1,2} - y_{4,2})^{2} = \xi_{2,4}^2,$$
$$(x_{1,1} - y_{5,1})^{2} + (x_{1,2} - y_{5,2})^{2} = \xi_{2,5}^2$$
or equivalently
$$x_{1,1}^{2} + x_{1,2}^{2} + 2x_{1,1} - 2x_{1,2} = -1, $$
$$x_{1,1}^{2} + x_{1,2}^{2} - 2x_{1,1} - 2x_{1,2} = 2.995, $$
$$x_{1,1}^{2} + x_{1,2}^{2} - 2x_{1,1} - 4x_{1,2} = 2.997.$$

Subtracing the first row from the second and the third we obtain

$$-4x_{1,1} = 3.995,$$
$$-4x_{1,1} - 2x_{1,2} = 3.997$$
and the solution of the last equation is $x_{1,1} = -0.998, x_{1,2} = 0.001$. Thus the approximate value of $x_{1}$ is $(-0.998, 0.001)$. In the same way we can approximate the point $x_{2}$.

\section*{Acknowledgments}

\hskip0.5cm The author would like to thank Professor Yosef Yomdin from Weizmann Institute of Science for his useful comments and suggestions during the writing of this article.

\section{Appendix}

\begin{lem}
Let $g$ be a function, defined on $\Bbb R^{+}$, such that its radial extension belongs to the Schwartz space $\mathrm{S}\left(\Bbb R^{n}\right)$. Then $g$ and its Hankel transform $G$ are continuously differentiable and belong to $L^{1}\left(\Bbb R^{+}, r^{\frac{n - 1}{2}}\right)$.
\end{lem}

\begin{proof}

Denote by $g_{0}$ and $G_{0}$ respectively the radial extensions of $g$ and $G$ to $\Bbb R^{n}$. Using a spherical coordinates system in $\Bbb R^{n}$, we have
$$\int_{|x| > 1}g_{0}(x)|x|^{-\frac{n - 1}{2}}dx = \int_{1}^{\infty}\int_{\Bbb S^{n - 1}}g_{0}(r\theta)r^{\frac{n - 1}{2}}d\theta dr = \frac{2\pi^{\frac{n}{2}}}{\Gamma\left(\frac{n}{2}\right)}\int_{1}^{\infty}g(r)r^{\frac{n - 1}{2}}dr.$$
Since $g_{0}$ belongs to $\mathrm{S}\left(\Bbb R^{n}\right)$ it follows that the integral in the left hand side converges and hence also the integral in the right hand side. Since $g$ is also bounded on the interval $[0,1]$ (since $g_{0}$ is bounded in the unit disk), it follows that $g$ is in $L^{1}\left(\Bbb R^{+}, r^{\frac{n - 1}{2}}dr\right)$. Denote by $\widehat{g}_{0}$ the Fourier transform of $g$, then using spherical coordinates again we have
$$\widehat{g}_{0}(\omega) = \frac{1}{(2\pi)^{\frac{n}{2}}}\int_{\Bbb R^{n}}g_{0}(x)e^{-i\langle \omega, x\rangle}dx = \frac{1}{(2\pi)^{\frac{n}{2}}}\int_{0}^{\infty}\int_{\Bbb S^{n - 1}}g_{0}(r\theta)e^{-ir\langle \omega, \theta\rangle}d\theta r^{n - 1}dr.$$
Denoting $\omega = \lambda\psi$, where $\lambda = |\omega|, \psi = \frac{\omega}{|\omega|}$
and using the identity
$$\int_{\Bbb S^{n - 1}}e^{-it\left\langle \psi, \theta\right\rangle} d\theta = (2\pi)^{\frac{n}{2}}j_{\frac{n}{2} - 1}(t)$$
it follows that
$$\widehat{g}_{0}(\lambda\psi) = \int_{0}^{\infty}g(r)j_{\frac{n}{2} - 1}(\lambda r)r^{n - 1}dr = G(\lambda).$$
Hence $\widehat{g}_{0}$ is the radial extension of $G$. Since the Fourier transform maps the Schwartz space $\mathrm{S}\left(\Bbb R^{n}\right)$ onto itself and since $g_{0}$ is in $\mathrm{S}\left(\Bbb R^{n}\right)$, it follows that the same is true for the radial extension of $G$. Hence, from the same arguments
we used for $g$ and its radial extension, it follows that $G$ belongs to $L^{1}\left(\Bbb R^{+}, r^{\frac{n - 1}{2}}dr\right)$.

The assertion that $g$ is continuously differentiable follows immediately from the fact that its radial extension if infinitely differentiable and the same is true for $G$.
\end{proof}

\begin{lem}

Let $n\geq 2$ and $\lambda_{1},...,\lambda_{n}$ be $n$ real numbers satisfying $\lambda_{i}\neq\lambda_{j}, i\neq j$. Let $\sigma:\{1,2,...,n\}\rightarrow\{1,2,...,n\}$ be a permutation which is different from the identity. Define the following matrix

\begin{equation}
    M =  \left(
            \begin{array}{cccc}
                \lambda_{1} - \lambda_{\sigma(1)} & \lambda_{2} - \lambda_{\sigma(2)} & ... & \lambda_{n} - \lambda_{\sigma(n)}\\
                \lambda_{1}^{2} - \lambda_{\sigma(1)}^{2} & \lambda_{2}^{2} - \lambda_{\sigma(2)}^{2} & ... & \lambda_{n}^{2} - \lambda_{\sigma(n)}^{2}\\
                .................\\
                \lambda_{1}^{n} - \lambda_{\sigma(1)}^{n} & \lambda_{2}^{n} - \lambda_{\sigma(2)}^{n} & ... & \lambda_{n}^{n} - \lambda_{\sigma(n)}^{n}
            \end{array}
    \right).
\end{equation}
Then, if $v = (v_{1},...,v_{n})\in\Bbb R^{n}$ satisfies $Mv^{T} = 0$ then there are two indices $i$ and $j$, $i\neq j$ such that $v_{i} = v_{j}$.
\end{lem}

\begin{proof}
The proof is by induction on $n\geq 2$. Since $\sigma$ is different from the identity then for $n = 2$ we have $\sigma(1) = 2$ and $\sigma(2) = 1$. In this case the vector $v = (1,1)$ is in the kernel of the matrix $M$ as defined by (5.1) and indeed this vector has two equal components with different indices. There are no other independent vectors in the kernel of $M$ since otherwise $M = 0$ which will imply that $\lambda_{1} = \lambda_{2}$.

Now we assume the induction hypothesis for every integer $m$ satisfying $2\leq m \leq n - 1$ and our aim is to prove it for the integer $n$ where we assume that $n\geq 3$.

First observe that we can assume that $\sigma$ does not have any fixed points. Indeed, suppose without loss of generality that $\sigma(n) = n$, then the matrix $M$ has the following form
$$ M =  \left(
            \begin{array}{ccccc}
                \lambda_{1} - \lambda_{\sigma(1)} & \lambda_{2} - \lambda_{\sigma(2)} & ... & \lambda_{n - 1} - \lambda_{\sigma(n - 1)} & 0\\
                \lambda_{1}^{2} - \lambda_{\sigma(1)}^{2} & \lambda_{2}^{2} - \lambda_{\sigma(2)}^{2} & ... & \lambda_{n - 1}^{2} - \lambda_{\sigma(n - 1)}^{2} & 0\\
                .................\\
                \lambda_{1}^{n - 1} - \lambda_{\sigma(1)}^{n - 1} & \lambda_{2}^{n - 1} - \lambda_{\sigma(2)}^{n - 1} & ... & \lambda_{n - 1}^{n - 1} - \lambda_{\sigma(n - 1)}^{n - 1} & 0\\
                \lambda_{1}^{n} - \lambda_{\sigma(1)}^{n} & \lambda_{2}^{n} - \lambda_{\sigma(2)}^{n} & ... & \lambda_{n - 1}^{n} - \lambda_{\sigma(n - 1)}^{n} & 0
            \end{array}
    \right) =
    \left(
        \begin{array}{cc}
            A & \overline{0}\\
            * & 0
        \end{array}
    \right),$$
where $A$ is the $(n - 1)\times(n - 1)$ top left sub matrix of $M$. Hence, if $Mv^{T} = 0$ then $v$ has the form $v = (v_{0}, \ast)$ where $v_{0}\in\Bbb R^{n - 1}$ is in the kernel of $A$. Now we can use our induction hypothesis on the matrix $A$. Indeed, if $\widetilde{\sigma}$ denotes the restriction of $\sigma$ to the set $\{1,2,...,n - 1\}$ then since $\sigma(n) = n$ it follows that $\widetilde{\sigma}(\{1,...,n - 1\}) = \{1,...,n - 1\}$ and since $\sigma$ is different from the identity then the same is true for $\widetilde{\sigma}$. Hence, using the induction hypothesis on the matrix $A$ is follows that $v_{0}$ has at least two equal components with different indices and thus the same is true for $v$.

Now, with the assumption that $\sigma$ does not have any fixed points it follows that the set $\{1,2,...,n\}$ has a partition into mutually disjoint sets $A_{1},...,A_{k}$ which satisfy the following conditions:\\

(i) $\left|A_{i}\right|\geq 2, 1\leq i\leq k$.

\vskip0.3cm

(ii) $\sigma\left(A_{i}\right) = A_{i}, 1\leq i\leq k$.

\vskip0.3cm

(iii) For every $1\leq i\leq k$ and every non-empty subset $B\subsetneqq A_{i}$ we have $\sigma(B)\neq B$.\\

We will distinguish between the cases when $k = 1$ and $k \geq 2$.

\subsection{The Case $k = 1$:}

If $k = 1$ we will show that the kernel of $M$ is spanned by the vector $v = (1,1,...,1)$. Indeed, suppose that $u = (u_{1},...,u_{n})$ is in the kernel of $M$, we can assume without loss of generality that $u_{n} = 0$ (since otherwise we can replace $u$ with $u - u_{n}v$). We will show that $u = \overline{0}$. Since $u$ is in the kernel of $M$ we have
$$ Mu =  \left(
            \begin{array}{ccccc}
                \lambda_{1} - \lambda_{\sigma(1)} & ... & \lambda_{n - 2} - \lambda_{\sigma(n - 2)} & \lambda_{n - 1} - \lambda_{\sigma(n - 1)} & \lambda_{n} - \lambda_{\sigma(n)}\\
                \lambda_{1}^{2} - \lambda_{\sigma(1)}^{2} & ... & \lambda_{n - 2}^{2} - \lambda_{\sigma(n - 2)}^{2} & \lambda_{n - 1}^{2} - \lambda_{\sigma(n - 1)}^{2} & \lambda_{n}^{2} - \lambda_{\sigma(n)}^{2}\\
                .................\\
                \lambda_{1}^{n} - \lambda_{\sigma(1)}^{n} & ... & \lambda_{n - 2}^{n} - \lambda_{\sigma(n - 2)}^{n} & \lambda_{n - 1}^{n} - \lambda_{\sigma(n - 1)}^{n} & \lambda_{n}^{n} - \lambda_{\sigma(n)}^{n}
            \end{array}
    \right)
    \left(
        \begin{array}{c}
        u_{1}\\
        ...\\
        u_{n - 2}\\
        u_{n - 1}\\
        0
        \end{array}
    \right)
$$
\begin{equation}
= \left(\rho_{1} - \rho_{\sigma(1)}\right)u_{1} + ... + \left(\rho_{n - 2} - \rho_{\sigma(n - 2)}\right)u_{n - 2} + \left(\rho_{n - 1} - \rho_{\sigma(n - 1)}\right)u_{n - 1} = \overline{0}
\end{equation}
where
\begin{equation}
 \rho_{i} = \left(\lambda_{i},\lambda_{i}^{2}...,\lambda_{i}^{n}\right)^{T}, 1 \leq i \leq n.
\end{equation}
\vskip0.2cm

For the case $k = 1$ condition (iii) implies that $\sigma\left(\{1,2,...,n - 1\}\right) \neq \{1,2,...,n - 1\}$. Hence, we can assume without loss of generality that $\sigma(k)\neq n - 1, 1 \leq k \leq n - 1$. Hence the vectors $\rho_{n - 1}$ and $\rho_{n}$ appear only once in the brackets of equation (5.2) and thus from equation (5.2) we have
\begin{equation}
a_{1}\rho_{1} + ... + a_{n - 2}\rho_{n - 2} + u_{n - 1}\rho_{n - 1} - u_{\sigma^{- 1}(n)}\rho_{n} = \overline{0}
\end{equation}
where $a_{1},...,a_{n - 2}$ are constants which depend on $u_{1},...,u_{n - 1}$. Now observe that if non of the vectors $\rho_{1},...,\rho_{n}$ is equal to zero then by equation (5.3) it follows that they are linearly independent since they form the $n$ columns of an $n\times n$ Vandermonde matrix and by our assumption $\lambda_{i}\neq\lambda_{j}$ if $i\neq j$. If one of these vectors is equal to zero then the other are different from zero and thus they are linearly independent since they form a sub matrix of the Vandermonde matrix. Hence since at least one of the vectors $\rho_{n - 1}$ or $\rho_{n}$ is different from zero it follows from equation (5.4) that either $u_{n - 1} = 0$ or $u_{\sigma^{-1}(n)} = 0$.

Continuing in this way, assume that after $n - k - 1$ steps we can prove that $n - k - 1$ coefficients from the set $\left\{u_{1},...,u_{n - 1}\right\}$ are equal to zero. Then from equation (5.2) we will have
\begin{equation}
\left(\rho_{i_{1}} - \rho_{\sigma(i_{1})}\right)u_{i_{1}} + ... + \left(\rho_{i_{k - 1}} - \rho_{\sigma(i_{k - 1})}\right)u_{i_{k - 1}} + \left(\rho_{i_{k}} - \rho_{\sigma(i_{k})}\right)u_{i_{k}} = \overline{0}.
\end{equation}
Since $\sigma\left(\left\{i_{1},...,i_{k}\right\}\right) \neq \left\{i_{1},...,i_{k}\right\}$ we can assume, without loss of generality, that $\sigma(i_{p})\neq i_{k}, 1\leq p\leq k$. Also, there exists an index $l$, different from any of the indices $i_{1},...,i_{k}$, such that $l\in\sigma\left(\left\{i_{1},...,i_{k}\right\}\right)$. Hence the vectors $\rho_{i_{k}}$ and $\rho_{l}$ appear only once in the brackets of equation (5.5) and thus from equation (5.5) we have
$$b_{1}\rho_{1} + ... + b_{i_{k} - 1}\rho_{i_{k} - 1} + u_{i_{k}}\rho_{i_{k}} + b_{i_{k} + 1}\rho_{i_{k + 1}} + ... + b_{l - 1}\rho_{l - 1} - u_{\sigma^{-1}(l)}\rho_{l} + b_{l + 1}\rho_{l + 1} + ... + b_{n}\rho_{n} = \overline{0}$$
where $b_{j}, 1\leq j\leq n, j\neq i_{k}, j\neq\sigma^{-1}(l)$ are constants which depend on $u_{i_{1}},..., u_{i_{k}}$. Since at least one of the vectors $\rho_{l}$ or $\rho_{i_{k}}$ is different from zero we can conclude, exactly as before, that either $u_{i_{k}} = 0$ or $u_{\sigma^{-1}(l)} = 0$.

Hence after $n - 1$ steps we can conclude that $u_{1} = 0,...,u_{n - 1} = 0$.

\subsection{The Case $k \geq 2$:}

For the case $k\geq 2$ the set $\{1,2,...,n\}$ has a partition into disjoints subsets $A_{1},...,A_{k}$ which satisfy conditions (i), (ii) and (iii). Since at most one of the vectors $\rho_{1},...,\rho_{n}$ is equal to zero and $k\geq2$ we can assume, without loss of generality, that all the vectors with indices in the set $A_{1}$ are different from zero and again without loss of generality we can assume that
$$A_{1} = \{1,...,m\}, 2\leq m\leq n - 2.$$
Now assume again that $u = (u_{1},...,u_{n})$ is in the kernel of $M$, we will show that there are two different indices, $i$ and $j$, such that $u_{i} = u_{j}$. From the definition of the matrix $M$ it follows that
$$\left(\rho_{1} - \rho_{\sigma(1)}\right)u_{1} + ... + \left(\rho_{m} - \rho_{\sigma(m)}\right)u_{m} + \left(\rho_{m + 1} - \rho_{\sigma(m + 1)}\right)u_{m + 1} + ... + \left(\rho_{n} - \rho_{\sigma(n)}\right)u_{n} =\overline{0}.$$
The last equation can be rewritten as
\begin{equation}
(u_{1} - u_{\sigma^{-1}(1)})\rho_{1} + ... + (u_{m} - u_{\sigma^{-1}(m)})\rho_{m} + c_{m + 1}\rho_{m + 1} + ... + c_{n}\rho_{n} = \overline{0}
\end{equation}
where $c_{m + 1},...,c_{n}$ are constants which depend on $u_{m + 1},...,u_{n}$. Since, by our assumption, all the vectors $\rho_{1},...,\rho_{m}$ are different from zero, they form a sub matrix of the Vandermonde matrix which implies that they are linearly independent. Hence from equation (5.6) we have
\begin{equation}
u_{1} - u_{\sigma^{-1}(1)} = 0,..., u_{m} - u_{\sigma^{-1}(m)} = 0.
\end{equation}
We will now show that the vector $\widetilde{u} = (u_{1},...,u_{m},0,...,0)$ is also in the kernel of $M$. Indeed, from the definition of $M$ we have to prove that
$$ \left(\rho_{1} - \rho_{\sigma(1)}\right)u_{1} + ... + \left(\rho_{m} - \rho_{\sigma(m)}\right)u_{m}$$ $$ =
\rho_{1}u_{1} + ... + \rho_{m}u_{m} - \underset{(\ast)}{\underbrace{\rho_{\sigma(1)}u_{1} - ... - \rho_{\sigma(m)}u_{m}}} = \overline{0}.$$
Now by condition (ii) it follows that $\sigma\left(\{1,...,m\}\right) = \{1,...,m\}$ and thus we also have
$\sigma^{-1}\left(\{1,...,m\}\right) = \{1,...,m\}$. Thus the set $\{1,...,m\}$ is a partition of $\{\sigma^{-1}(1),...,\sigma^{-1}(m)\}$. Using this observation on $(\ast)$ we have
$$\rho_{1}u_{1} + ... + \rho_{m}u_{m} - \rho_{\sigma(1)}u_{1} - ... - \rho_{\sigma(m)}u_{m}$$ $$ =
\rho_{1}u_{1} + ... + \rho_{m}u_{m} - \rho_{\sigma(\sigma^{-1}(1))}u_{\sigma^{-1}(1)} - ... - \rho_{\sigma(\sigma^{-1}(m))}u_{\sigma^{-1}(m)}$$
$$ = \rho_{1}u_{1} + ... + \rho_{m}u_{m} - \rho_{1}u_{\sigma^{-1}(1)} - ... - \rho_{m}u_{\sigma^{-1}(m)}$$
$$ = \left(u_{1} - u_{\sigma^{-1}(1)}\right)\rho_{1} + ... + \left(u_{m} - u_{\sigma^{-1}(m)}\right)\rho_{m}$$
and the last expression is equal to zero by (5.7).

Since $\widetilde{u}$ is in the kernel of $M$ it follows that the vector $(u_{1},...,u_{m})$ is in the kernel of the top left $m\times m$ sub matrix of $M$. That is

\begin{equation}
    \left(
            \begin{array}{ccc}
                \lambda_{1} - \lambda_{\sigma(1)} & ... & \lambda_{m} - \lambda_{\sigma(m)}\\
                .................\\
                \lambda_{1}^{m} - \lambda_{\sigma(1)}^{m} & ... & \lambda_{m}^{m} - \lambda_{\sigma(m)}^{m}
            \end{array}
    \right)
    \left(
        \begin{array}{c}
            u_{1}\\
            ...\\
            u_{m}
        \end{array}
    \right)
     =
     \left(
        \begin{array}{c}
            0\\
            ...\\
            0
        \end{array}
    \right).
\end{equation}
Now we can use the induction hypothesis on equation (5.8). Indeed, since $\sigma\left(\{1,...,m\}\right) = \{1,...,m\}$ and $\lambda_{i}\neq\lambda_{j}$ for $1\leq i,j\leq m, i\neq j$, all of the conditions for the sub matrix in the left hand side of equation (5.8) are satisfied. Hence we conclude that there are two different indices, $i$ and $j$ with $1\leq i,j\leq m$, such that $u_{i} = u_{j}$. This in particular implies that the vector $u$ has two equal components with different indices. Thus the case $k\geq 2$ is also proved.

\end{proof}

\end{document}